\documentclass{amsart}
\usepackage{amssymb,latexsym,graphicx,psfrag,epstopdf,hyperref,color,pinlabel,enumitem}

\usepackage[all,arc,dvips]{xy}

\newcommand{\executeiffilenewer}[3]{%
\ifnum\pdfstrcmp{\pdffilemoddate{#1}}%
{\pdffilemoddate{#2}}>0%
{\immediate\write18{#3}}\fi%
}

\def\frac#1#2{{#1\over#2}}

\newcommand{\ZZ}{\mathbb{Z}} 

\newcommand{\QQ}{\mathbb{Q}}
 
\newcommand{\RR}{\mathbb{R}}

\def\vol{\mathrm{Vol}}

\def\OO{\mathcal{O}} 
\def\SS{\mathcal{S}}

\def\Vol{\mathrm{Vol}}

\def\TT{\mathcal{T}} 
 
\def\QQ{\mathcal{Q}}

\def\TT{\mathcal{T}}

\def\QQ{\mathcal{Q}}

\def\split{\setminus \!\! \setminus}

\newtheorem{proposition}{Proposition}[section]
\newtheorem{theorem}[proposition]{Theorem}
\newtheorem{definition}[proposition]{Definition}

\newtheorem{corollary}[proposition]{Corollary}

\newtheorem{lemma}[proposition]{Lemma}
\newtheorem{characteristic}[proposition]{Characteristic Suborbifold}

\theoremstyle{remark}

\theoremstyle{remark}

\newtheorem*{acknowledgments}{Acknowledgments}

\numberwithin{equation}{section}

\def\Vol{\mbox{\rm{Vol}}}

\begin{document}
\Large
\title[Guts and volume for hyperbolic $3$--orbifolds]{Guts and volume for hyperbolic $3$--orbifolds with underlying space
$S^3$}
\author[]{Christopher~K.~Atkinson}
\address{Department of Mathematics, University of Minnesota Morris, Morris, MN
56267, USA}
\email{catkinso@morris.umn.edu}
\author[]{Jessica Mallepalle}
\address{Department of Mathematics, Arcadia University, Glenside, PA
19038, USA}
\email{jmallepalle@arcadia.edu}
\author[]{Joseph Melby}
\address{Department of Mathematics, University of Minnesota Morris, Morris, MN
56267, USA}
\email{melby131@morris.umn.edu}
\author[]{Shawn~Rafalski}
\address{Department of Mathematics, Fairfield University, Fairfield, CT 06824, USA}
\email{srafalski@fairfield.edu}
\author[]{Jennifer Vaccaro}
\address{Franklin W. Olin College of Engineering, Needham, MA 
02492, USA}
\email{jennifer.vaccaro@students.olin.edu}
\date{\noindent March 2017. 
}
\thanks{This work is supported by a grant from the National Science Foundation DMS-1358454.}

\begin{abstract}
For a hyperbolic $3$--orbifold with underlying space the $3$--sphere,
we obtain a lower bound on its volume in the case that it contains an
essential $2$--suborbifold with underlying space the $2$--sphere with four
cone points. Our techniques involve computing the guts of the orbifold split
along the $2$--suborbifold via a careful analysis of its topology. We also
characterize the orbifolds of this type that have empty guts.
\end{abstract}
\maketitle

\section{Introduction}\label{S:Intro} 
This paper contributes to understanding the
organization of the volumes of hyperbolic $3$-manifolds and $3$-orbifolds.
One common theme in this organization is the classification of such spaces
for which the presence of a particular type of 2-dimensional sub-object
informs the volume of the ambient space. In the case we are considering, an
embedded incompressible 2-suborbifold in a 3-orbifold (one measure of the
3-orbifold's topological complexity) either informs us about the volume of
this 3-orbifold or about its topological structure. Examples of this
volume/complexity dynamic occur in recent years in the work of Gabai,
Meyerhoff and Milley in the identification of the Weeks--Fomenko--Matveev
manifold as the lowest volume hyperbolic $3$-manifold \cite{GabMeyMill1,
GabMeyMill2}. Similarly, Gehring, Marshall, and Martin have identified the
lowest volume hyperbolic $3$-orbifolds \cite{GehringMartin, MarshallMartin}
based, in part, on this idea. Miyamoto gives a lower volume bound for
hyperbolic 3-manifolds \cite{miyamoto} (generalized to 3-orbifolds in
\cite{rafalski-thesis}) with totally geodesic boundary. Recent work by a
subset of the authors identifies the lowest volume polyhedral hyperbolic
3-orbifolds that contain an arbitrary essential 2-suborbifold
\cite{SmallestHakenPoly}.  

In the current paper, we employ a result of Agol,
Storm, and Thurston \cite{agol-storm-thurston} to find lower volume bounds
(or else, a topological characterization) for a large class of hyperbolic
3-orbifolds that contain a particular type of essential 2-suborbifold (i.e. \emph{Haken} 
3-orbifolds) in
terms of the topology of that suborbifold.

Let $\Vol(\cdot)$ denote hyperbolic volume, $\chi(\cdot)$ Euler
characteristic, and $V_{8} \approx 3.66$ the volume of the regular, ideal,
hyperbolic octahedron. Let $S^2(n_1,n_2,n_3,n_4)$ denote the orientable $2$--orbifold
with base space $S^2$ and cone points of orders $n_1$, $n_2$, $n_3$, and
$n_4$.  We prove the following theorem:

\begin{theorem}\label{T:VolumeGuts} Let $\OO$ be a compact, orientable,
irreducible, turnover-reduced, $3$-orbifold with underlying space $S^{3}$ and
singular set $\Sigma$ whose interior admits a hyperbolic structure of finite
volume. Suppose $\OO$ contains an incompressible $2$-suborbifold $\SS$ 
of type $S^2(n_1,n_2,n_3,n_4)$.
Then one of the following lower bounds for $\Vol(\OO)$ holds:
	\begin{enumerate}
		\item $\Vol(\OO) \geq -V_{8}\chi(\SS) = V_{8} \left(2 -			
			\sum_{i=1}^{4} 1/n_{i}  \right)$, or \item $\Vol(\OO) \geq
			  \frac{1}{2}V_{8} \left(-\chi(\SS) + 1- 1/n_{i_{1}} -
			  1/n_{i_{2}}  \right)$, 
			\\ 
			(where $\{n_{i_{1}}, n_{i_{2}} \} \subset \{n_{1}, n_{2}, n_{3},
		  n_{4}\}$),  or \item $\Vol(\OO) \geq -\frac{1}{2}V_{8}\chi(\SS)$,
			or 
		\item $\Vol(\OO) \geq \frac{1}{2}V_{8} \left( 2- 1/n_{i_{1}} -
			  1/n_{i_{2}} - 1/n_{i_{3}} - 1/n_{i_{4}} \right)$, 
			\\
			(where $\{n_{i_{1}}, n_{i_{2}},n_{i_{3}}, n_{i_{4}} \} \subset
		  \{n_{1}, n_{2}, n_{3}, n_{4}\}$),  or \item $\Vol(\OO) \geq
			\frac{1}{2}V_{8} \left(1- 1/n_{i_{1}} - 1/n_{i_{2}}  \right)$, 
			\\ 
			(where $\{n_{i_{1}}, n_{i_{2}} \} \subset \{n_{1}, n_{2}, n_{3},
		  n_{4}\}$),  or \item $\OO$ has one of the forms given in Section
			\ref{s:emptyguts}.
	\end{enumerate} 
\end{theorem}

Note that we mean something more general than the standard use of hyperbolic
structure of finite volume. We allow not only for cusps with Euclidean
cross--sections, but for totally geodesic boundary components. See the end of
Section~\ref{s:orbifolds} for details.

Let $D^2_*(n_1,n_2)$ denote the nonorientable $2$-orbifold with base
space $D^2$ with mirrored boundary and with interior containing cone points of
orders $n_1$ and $n_2$. Using the fact that an incompressible $2$-suborbifold in $\OO$ of 
this form has a regular neighborhood whose boundary is an orientable 
and incompressible $S^2(n_1,n_1,n_2,n_2)$, the next result follows 
immediately from the above theorem:

\begin{corollary}\label{C:VolumeGuts}
Under the same conditions on $\OO$ as above, suppose $\OO$ contains an incompressible $2$-suborbifold $\SS$ of type $D^2_*(n_1,n_2)$. Then one of the following lower bounds for $\Vol(\OO)$ holds:
\begin{enumerate}
	\item $\Vol(\OO) \geq -V_{8}\chi(\SS) = V_{8} \left(1 - 1/n_{1} -
			1/n_{2}  \right)$, or \item $\Vol(\OO) \geq \frac{1}{2}V_{8}
			  \left(1- 1/n_{i_{1}} - 1/n_{i_{2}}  \right)$, 
			\\ 
			(where $\{n_{i_{1}}, n_{i_{2}} \} \subset \{n_{1}, n_{2}\}$),  or
	\item $\OO$ has one of the forms given in Section \ref{s:emptyguts}.
\end{enumerate}

\end{corollary}

Note that the smallest possible lower bounds given by
Theorem~\ref{T:VolumeGuts} and Corollary~\ref{C:VolumeGuts} is $\frac{1}{12}
v_8 \approx 0.305$.
This volume bound occurs in number (5) of Theorem~\ref{T:VolumeGuts} and in
number (2) of Corollary~\ref{C:VolumeGuts} when $n_{i_1} = 2$ and $n_{i_2} =
3$. As a consequence, one can conclude that under the hypotheses, if
$\vol{\OO} < \frac{1}{12} v_8$, then $\OO$ has empty guts and is of the form
described in Proposition~\ref{P:hungry}.

\subsection{Organization} In Section~\ref{s:background}, we give the relevant
background on orbifolds and describe the conventions that we will use
throughout the paper. We also describe the characteristic suborbifold theory
and relevant ideas. In Section~\ref{s:proof} we prove a lemma that describes
how essential annuli can arise in certain $3$--orbifolds and prove the main
theorem of the paper. In Section~\ref{s:hungry}, we characterize orbifolds
having empty guts.

\begin{acknowledgments}
The authors  give special thanks to Ian Agol and Peter Shalen for helpful conversations. Much of the work in this paper was done as part of the 2015 Fairfield University REU Program, sponsored by the National Science Foundation.
\end{acknowledgments}

\section{Background and Definitions}\label{s:background}

We recall some necessary facts about orbifolds here, and refer the reader to
several excellent resources \cite{BMP03-1,CoopHodgKer00}. An $n$--orbifold is a
generalization of the notion of $n$--manifold that allows for local neighborhoods
to be modeled on the quotient of $\RR^n$ by a (possibly trivial) finite group
acting properly discontinuously. In many cases, $n$--orbifolds arise as the
quotient of a manifold by a finite group of symmetries.

\subsection{Orbifolds}\label{s:orbifolds}

An orientable \emph{$3$--orbifold} $\OO$ is a pair
$(X_{\OO}, \Sigma_{\OO})$ where $X_{\OO}$ is an orientable $3$--manifold and
$\Sigma_{\OO}$ is an embedded graph in $X_{\OO}$ with edges labeled by
integers $n\geq 2$.  The manifold $X_{\OO}$ is called the  \emph{underlying
space} of the orbifold $\OO$. The graph $\Sigma_{\OO}$ is called the \emph{singular
locus} of the orbifold. The graph $\Sigma_{\OO}$ need not be connected and may
contain components consisting of single loops (while multi-edges between vertices 
are allowed, they tend to violate the geometric assumptions we will use in this paper). 
In the case where $X_{\OO}$ is closed, $\Sigma_{\OO}$ must be a trivalent
graph. If $X_{\OO}$ has boundary, then $\Sigma_{\OO}$ may also have univalent
vertices on the boundary of $X_{\OO}$. The labeling on an edge of the
singular locus is called the \emph{order} of the singular locus along that
edge.  

The data $(X_\OO,
\Sigma_\OO)$ completely describe the orbifold.  Neighborhoods $U \subset
(X_\OO \setminus \Sigma_\OO)$ of points are modeled on $\RR^3$. Neighborhoods
$U_x$ of points $x \in \Sigma_\OO$ are modeled on $\RR^3/{G_x}$ where $G_x$
is a finite subgroup of $O(3)$. In the case where $x$ lies in an edge of
$\Sigma$ labeled $n$, $G_x \approx \ZZ_n$.  If $x$ is a vertex of $\Sigma$
meeting edges labeled $p$, $q$, and $r$, then $G_x$ is the spherical triangle
group generated by rotations of order $p$, $q$, and $r$.  Note that this
implies that $\frac1p + \frac1q + \frac1r >1$. 

We will abuse notation in this paper and allow for the
labeling of the singular locus to violate the $\frac1p + \frac1q + \frac1r
>1$ condition at a vertex. We will refer to a vertex of the singular locus as
a \emph{spherical, rigid Euclidean, or hyperbolic vertex} if $\frac1p +
\frac1q + \frac1r$ is greater than, equal to, or less than $1$, respectively.
A hyperbolic or rigid Euclidean vertex encodes a boundary component of the
orbifold. If $\frac1p + \frac1q + \frac1r \leq 1$, then the orbifold has a
boundary component obtained by doubling a triangle with angles $\pi/p$,
$\pi/q$, and $\pi/r$ with edges of the singular locus meeting the
corresponding cone points on this boundary component. The reason for this
abuse of notation is that it turns out that these ``vertices'' can usually be
treated the same, regardless of whether they are spherical, rigid Euclidean,
or hyperbolic. For this reason, when we say that an orbifold has underlying
space $S^3$, it may be the case that there are boundary components coming
from the non--spherical vertices.

\subsection{Suborbifolds}

We say that $\SS$ is an \emph{orientable $2$--dimensional suborbifold} of an
orientable $3$--orbifold $\OO$ if $\SS$ is an embedded, orientable
topological surface such $\SS \cap \Sigma$ is either empty or is a
$0$--dimensional subset of the edges of $\Sigma$. Such an $\SS$ may be given
the structure of a $2$--orbifold. The \emph{underlying space} of the orbifold
structure is the topological surface itself. The \emph{singular locus} is the
collection of points of $\SS \cap \Sigma$ labeled by the same integers as
$\Sigma$. These points are locally modeled on $\RR^2$ modulo the action of
rotation of the corresponding order. A \emph{nonorientable $2$--suborbifold}
$\SS$ of $\OO$ is an embedded topological surface in $\OO$ such that the
boundary of a regular neighborhood of $\SS$ is a connected, orientable
$2$--suborbifold of $\OO$. The singular locus of a nonorientable
$2$--suborbifold includes points labeled by integers along with arcs along
which $\SS$ is locally modeled on $\RR^2$ modulo the action of reflection.

Just as in the case of manifolds, orbifolds have fundamental groups denoted
by $\pi_{1}(\cdot)$ (corresponding to the groups acting on their
universal covers), and a $2$--suborbifold of a 3-orbifold is called
\emph{incompressible} if its fundamental group injects into the fundamental
group of its ambient space and \emph{essential} if it is incompressible but
not isotopic to a boundary component. 

A \emph{turnover} is a $2$--suborbifold
with underlying space a $2$--sphere that meets $\Sigma$ in three points. These
turnovers are called \emph{spherical, Euclidean (rigid),} or
\emph{hyperbolic} if their three point labels satisfy (respectively) the
above conditions for the correspondingly-named vertices of $\Sigma$.  A
\emph{turnover-reduced} 3-orbifold is one in which every embedded turnover
encloses the cone on a vertex of $\Sigma$, and a $3$--orbifold is
\emph{irreducible} when every embedded $2$--suborbifold that is the quotient of
a $2$--sphere by a discrete group bounds a quotient of a $3$--ball by the same
discrete group. 

One key definition we will use involves splitting a $3$-orbifold along an embedded 
$2$-suborbifold. This yields an orbifold with boundary, in the sense mentioned in 
Subsection \ref{s:orbifolds}. We denote by $\OO \split \SS$ the path metric completion of $\OO
\setminus \SS$. If $\SS$ is an embedded orientable $2$--suborbifold of $\OO$,
then $\OO \split \SS$ consists of two $3$--orbifolds with boundary equal
to $\SS$. 

When a $3$--orbifold has rigid and hyperbolic vertices, we consider
these as boundary components by deleting an open neighborhood of the
vertices. In this case, we say a $3$--orbifold has a \emph{hyperbolic structure}
if deleting the rigid boundaries yields a space with  a complete hyperbolic
geometric structure in which each rigid vertex becomes a Euclidean (rigid)
cusp and in which the hyperbolic vertices become totally geodesic boundary
components in the geometric structure. 

Important $2$--suborbifolds we will consider are orbifold tori and annuli, which
are best thought of as quotients of standard tori and annuli by symmetries.
Figure \ref{F:OrbifoldAnnuli} depicts the different types of orbifold annuli.
In the figure, the symbol * indicates a part of the orbifold that has the
quotient structure associated to a reflection. These annuli can occur in a
$3$--orbifold in many ways. As an example, consider Figure \ref{F:essann1}$(i)$,
with the lone singular loop labeled by $e$ set equal to $2$: an orbifold annulus as in the lower
left of Figure \ref{F:OrbifoldAnnuli} appears as the annular region between
this loop and the equator of the spherical boundary of the orbifold depicted,
with the edge labeled 2 acting as an orbifold mirror for the annulus. An
(orientable) orbifold torus is either a topological torus or a topological
$2$--sphere with four singular points each labeled 2.

\begin{figure}
	\labellist
	\small\hair 2pt
	\pinlabel {$\partial$} [] at 85 205
	\pinlabel {$\partial$} [ ] at 85 108
	\pinlabel {$\partial$} [ ] at 85 79
	\pinlabel {$\ast$} [ ] at 49 12	
	\pinlabel {$\partial$} [ ] at 220 205
	\pinlabel {$2$} [ ] at 130 112
	\pinlabel {$2$} [ ] at 239 112
	\pinlabel {$\partial$} [ ] at 185 82
	\pinlabel {$\ast$} [ ] at 132 44
	\pinlabel {$\ast$} [ ] at 237 44
	\pinlabel {$2$} [ ] at 130 0
	\pinlabel {$\ast$} [ ] at 185 -3
	\pinlabel {$2$} [ ] at 239 0
	\pinlabel {$\partial$} [ ] at 286 205
	\pinlabel {$\ast$} [ ] at 332 160
	\pinlabel {$2$} [ ] at 263 112
  \endlabellist
  \centering
	\includegraphics[scale=1.0]{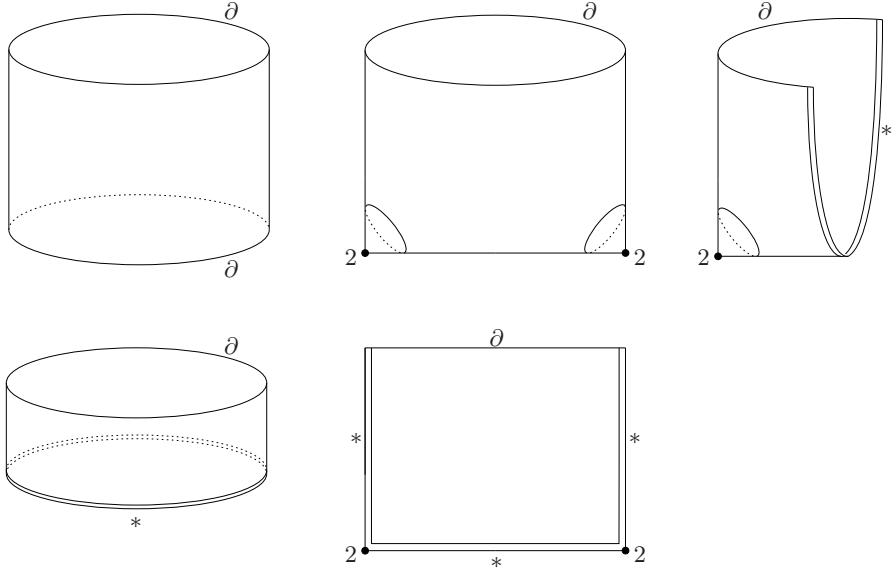}
	 \caption{The different types of orbifold annuli. The symbols $\partial$ and * indicate orbifold boundary and orbifold mirrors, respectively.}		
	\label{F:OrbifoldAnnuli}
\end{figure}

If a $2$--suborbifold is not incompressible, it is called
\emph{compressible}. In our case, this implies the existence of an orbifold
disk (a quotient of a disk by a symmetry) whose boundary lies on the
$2$--orbifold but that does not also bound an orbifold disk on the
$2$--orbifold. Examples of compressibility in the case of annuli are shown in
Figure \ref{F:AnnuliComps} (with the compressing disks in grey).
\begin{figure}
\labellist
\small\hair 2pt
    \pinlabel {$n$} [ ] at 58 62
    \pinlabel {$2$} [ ] at 173 42
    \pinlabel {$2$} [ ] at 218 42
    \pinlabel {$n$} [ ] at 201 68
	 \endlabellist
	    	\includegraphics{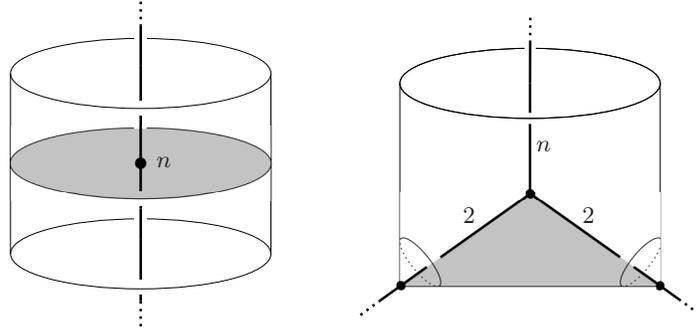}
	      \caption{Orientable compressible Annuli. In both cases, $n\geq
		  1$.}		
		\label{F:AnnuliComps}
\end{figure}

An \emph{orbifold rational tangle} is the space obtained by taking either (1)
a solid ball with two unknotted, integer-labeled strands that lie in its
interior except at their endpoints, which lie on the boundary of the ball; or
(2) a solid ball with five unknotted, integer-labeled strands that lie in its
interior in the shape of an ``H,'' except for the four endpoints of the
``H,'' which lie on the boundary of the ball; and performing an isotopy of
the boundary of the ball that permutes the four singular points. If we assume
a 3-orbifold $\OO$ satisfies the hypotheses of Theorem \ref{T:VolumeGuts},
then the turnover reducibility of $\OO$ implies the following simple
proposition. 

\begin{proposition}\label{P:incomp} 
\sloppy
Let $\OO$ be a compact, orientable,
  irreducible, turnover--reduced, $3$--orbifold with underlying space $S^3$.
  Suppose that $\SS = S^2(n_1,n_2,n_3,n_4)$ or $\SS= D_*^2(n_1,n_2)$ is a
  suborbifold of $\OO$. Then $\SS$ is incompressible in $\OO$ if and only if
  no component of the $\OO \split \SS$ is an orbifold rational tangle.
\end{proposition}

\begin{proof} Suppose $\SS = S^2(n_1,n_2,n_3,n_4)$ is compressible. Then a 
non-trivial loop $C$ on $\SS$ bounds an orbifold disk $D$ in (at least) one 
component of $\OO \split \SS$. Because $C$ is non-trivial on $\SS$, 
$C$ separates $\SS$ into two disks, each with two singular points of 
$\SS$. Attaching $D$ to these disks and using the 
irreducibility/turnover--reducibility of $\OO$ implies that this component of 
$\OO \split \SS$ is an orbifold rational tangle. If one of the components of 
$\OO \split \SS$ is an orbifold rational tangle, then 
this gives a compression of $\SS$ in $\OO$. 

If $\SS = D_*^2(n_1,n_2)$, then 
the same proof works, using the boundary of a regular neighborhood of 
$\SS$.
\end{proof}

\subsection{Pared orbifolds and guts}\label{s:guts}

We now introduce the definitions that will be central to the proof of our
theorem.

\begin{definition}\label{d:pared}
A \textbf{pared orbifold} is a pair $(\OO, P)$, where $\OO$ is a compact, orientable, irreducible 3-orbifold 
and $P \subset \partial \OO$ is a union of essential orbifold annuli and tori (possibly empty) such that 
		\begin{enumerate}
			\item every abelian, noncyclic subgroup of $\pi_{1}(\OO)$ is conjugate to a subgroup of 
					the fundamental group of a component of $P$, and
			\item every map of an orbifold annulus $(F, \partial F) \to (\OO, P)$ that is $\pi_{1}-$injective deforms, as a 
					map of pairs, into $P$.
		\end{enumerate}
$P$ is called the \textbf{parabolic locus} of $(\OO,P)$, and we define $\partial_{0}\OO$ to be
$\partial \OO -  \rm{int}(P)$. 
\end{definition}

Thurston proved (\cite[Theorem 6.5]{BMP03-1}, \cite{Boileau-Porti,
  Canary-McCullough, Kapovich,
Otal1, Otal2})  
that an oriented, turnover-reduced pared
orbifold with nonempty boundary (at least one component of which is not a
hyperbolic turnover) admits a geometrically finite hyperbolic structure. The
characteristic subpair theory for 3-manifolds \cite{Jaco-Shalen, Johansson}
holds in the category of 3-orbifolds \cite{Bonahon-Siebenmann}. In particular
(\cite[Remarks following Theorem 4.17]{BMP03-1}, \cite[Section 11, page
88]{Morgan}), we have the following characterization of the components of the
characteristic suborbifold:

\begin{characteristic}\label{C:char}
If $(\OO,P)$ is a pared orbifold with incompressible
$\partial_{0} \OO$, then there is a subpair $(N,S) \subset (\OO, \partial_{0}
\OO)$, which is a (possibly disconnected) suborbifold, uniquely determined up
to isotopy of pairs and whose components are of the following three types, up
to homeomorphism:

\begin{enumerate} 
  \item $(T \times I, \emptyset)$, a neighborhood of an
	  orbifold torus component $T \subset P$ (in our case, a neighborhood of
	a rigid cusp), or 
  \item  $(R,A)$, where $R$ is an orbifold solid torus
	and $A$ is a nonempty union of essential orbifold annuli in $\partial R$,
  or 
\item $(I{\rm-bundles}, \partial I{\rm-subbundles})$, with any
	$I$-bundles in the boundary of this type of component consisting of
	orbifold annuli that are not parallel into $\partial \OO$.
\end{enumerate} 
\end{characteristic}

For each component of type (2) above, $\overline{\partial R - A}$ is a union
of essential orbifold annuli in $(\OO, \partial_{0}\OO)$. We thicken these
and consider them as $I$-bundles over the annuli with their associated
$\partial I$-subbundles lying in $\partial_{0}\OO$, and add them to the
components of type (3) to form an $(I{\rm-bundle}, \partial
I{\rm-subbundle})$ pair $(W, \partial_{0} W) \subset (\OO, \partial_{0}\OO)$
called the \emph{window} of $(\OO, \partial_{0}\OO)$. The frontier of the
window $\partial_{1}W = \overline{\partial W - (\partial_{0}W \cup P)}$
consists of essential annuli in $(\OO, \partial_{0} \OO)$. Following Agol
\cite[page 3275]{Agol2cusp}, we observe that the pair 
\[(\overline{\OO - W},
\overline{\partial(\overline{\OO - W}) - \partial_{0} \OO})\]
is a pared
suborbifold of $\OO$ whose parabolic locus consists of these essential annuli
(up to isotopy). We have
the following:

\begin{definition}\label{d:guts}
The pared suborbifold $(\overline{\OO - W}, \overline{\partial(\overline{\OO
- W}) - \partial_{0} \OO})$ is denoted $\textbf{guts}(\OO, P)$.  If $\OO$ is
a compact, orientable 3-orbifold whose interior admits a hyperbolic metric of
finite volume, and $(X, \partial X) \subset (\OO, \partial \OO)$ is an
essential 2-orbifold, then we define $\textbf{guts}(X)$ to be \\ $guts(\OO
\split X, \partial \OO \split \partial X)$.
\end{definition} 

\noindent
We note that if $\partial_{1}W = \overline{\partial W - (\partial_{0}W \cup
P)}$ denotes the frontier of the window of $(\OO, \partial_{0}\OO)$, and if
$L$ is a component of $\overline{\OO - W}$, then either 
\begin{itemize}
	\item $L$ is a solid orbifold torus and $\partial_{1} W \cap L$ contains
	  at least one essential orbifold annulus on $\partial L$, or 
	\item all essential orbifold annuli in $(L, \overline{\partial_{0} \OO
	\cap L})$ are parallel into 
	\[(\partial_{1} W \cap L,
	\partial(\partial_{1}W \cap L)).\]
\end{itemize}
Components of the latter type are called \emph{pared acylindrical orbifolds},
and they admit hyperbolic metrics with totally geodesic boundary. 

A key tool we will use in this paper is the following (abbreviated here from
the complete, more powerful) result of Agol, Storm, and Thurston, in
combination with a result of Miyamoto, applied in the orbifold category (cf.
\cite[Theorem 9.1]{agol-storm-thurston}, \cite{miyamoto}):

\begin{theorem}\label{T:AST}
	Let $\OO$ be a compact 3-orbifold with interior a hyperbolic orbifold of
	finite volume, and let $\SS$ be an embedded incompressible 2-suborbifold
	in $\OO$. Then 
	\[\Vol(\OO) \geq -V_{8} \chi(guts(\SS)) = -\frac{V_{8}}{2}\chi(\partial
	(guts(\SS))).\]
\end{theorem}

\noindent
Accordingly, our strategy for the proof will be to identify the essential
annuli in $\OO \split \SS$ in order to effectively describe
$guts(\SS)$ and apply the volume bound of Theorem \ref{T:AST}.

\section{Proof of the Main Theorem}\label{s:proof}

In this section, we prove the main theorem.  Before doing so, we
need a technical lemma that describes how orientable, essential orbifold annuli can
arise in certain $3$-orbifolds. To prove the main theorem, this lemma will
be used to describe the possibilities for $guts(\SS)$, where $\SS$ is an
incompressible $2$--suborbifold of an orbifold as described in the hypotheses
of Theorem~\ref{T:VolumeGuts}.

\begin{lemma}\label{l:annuli}
Let $\QQ$ be a compact, orientable, irreducible, atoroidal, turnover-reduced $3$-orbifold
with underlying space $D^3$. Assume $\QQ$ has four singular points on $\partial D^{3}$ 
(labeled by their corresponding orders as $a$, $b$, $c$, and $d$), and let $\SS$ 
denote the $2$-orbifold $\partial D^{3}$ together with these four points.  
Let $P$ be the union of the rigid, Euclidean
boundary components and let $\partial_0 \QQ = \partial \QQ - P$.

Then at most one of the following holds:
\begin{enumerate}
  \item $(\QQ, \partial_0 \QQ)$ contains a single, essential, nonsingular
	annulus,
  \item $(\QQ, \partial_0 \QQ)$ contains a pair of essential, non--parallel
	$D^2(2,2)$ orbifold annuli, or
  \item $(\QQ, \partial_0 \QQ)$ contains a single essential, $D^2(2,2)$
	orbifold annulus.
\end{enumerate}
Moreover, the essential annuli are configured as in
Figure~\ref{F:essann1}.

\end{lemma} 

Note that this lemma doesn't preclude the presence of non--orientable
essential orbifold  annuli. However, the boundary of the regular neighborhood a
non--orientable orbifold annulus is an orientable annulus and will be in one
of the three cases described in the lemma.

\begin{figure}
  \labellist
  \small\hair 2pt
  \pinlabel {$a$} [ ] at 65 149
  \pinlabel {$b$} [ ] at 86 149
  \pinlabel {$R$} [ ] at 75 79
  \pinlabel {$c$} [ ] at 64 4
  \pinlabel {$d$} [ ] at 86 5
  \pinlabel {$e$} [ ] at 130 75
  \pinlabel {$(i)$} [ ] at 75 -8
  \pinlabel {$(ii)$} [ ] at 250 -8
  \pinlabel {$a$} [ ] at 243 148
  \pinlabel {$b$} [ ] at 268 148
  \pinlabel {$R_1$} [ ] at 250 118
  \pinlabel {$2$} [ ] at 225 90
  \pinlabel {$2$} [ ] at 277 90
  \pinlabel {$A_1$} [] at 335 100
  \pinlabel {$A_2$} [] at 335 55
  \pinlabel {$n$} [ ] at 250 80
  \pinlabel {$2$} [ ] at 223 63
  \pinlabel {$2$} [ ] at 277 63
  \pinlabel {$R_2$} [ ] at 250 36
  \pinlabel {$c$} [ ] at 243 7
  \pinlabel {$d$} [ ] at 268 8
				  \endlabellist
  \includegraphics[scale=0.95]{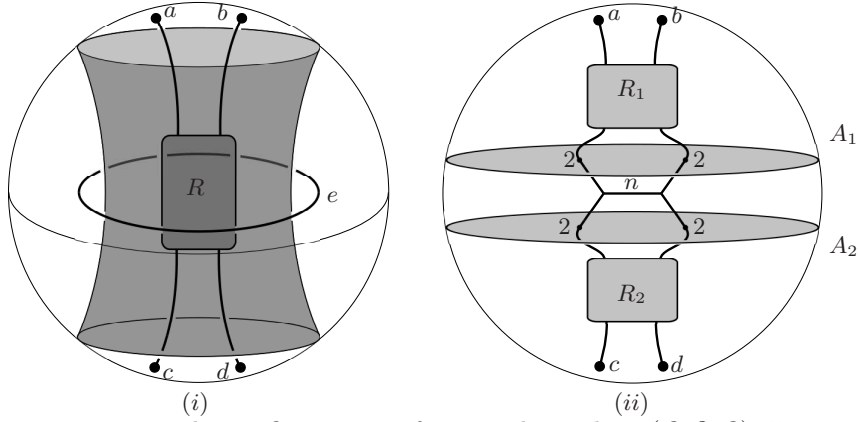}
		  \caption{The configurations of essential annuli in $(\QQ,
			\partial_{0} \QQ)$.  In $(i)$, neither of the sets
			$\{a,b\}$, $\{c,d\}$ can equal $\{2\}$. In $(ii)$, without
			loss of generality, $\{c,d\}$ cannot equal $\{2\}$, and $n \geq
			2$ unless $\{a,b\}$ = \{2\}, in which case $n=1$. The essential
		  annuli in $(ii)$ are $A_1$ and $A_2$. If $n = 1$, then $A_1$ and
		$A_2$ are parallel.} 
			\label{F:essann1}
\end{figure}

\begin{proof} 
  There are 5 types of orbifold annuli to consider, as in Figure
  \ref{F:OrbifoldAnnuli}.  The only components of $\partial_{0} \QQ$ are
  $\SS$ and any hyperbolic turnovers corresponding to the boundary components of $\QQ$.
  Because any closed loop on a hyperbolic turnover bounds an orbifold disk on
  the turnover, and because hyperbolic turnovers cannot have more than one
  cone point of order 2, the irreducibility of $\QQ$ implies that no
  essential annulus can have any part of its boundary carried in the
  hyperbolic turnover boundary components of $\QQ$. So we need only consider
  essential annuli with boundary contained in $\SS$. 
  
Let $A$ be such an essential annulus. If $A$ is nonsingular, then its two
boundary circles must consist of parallel curves on $\SS$ that separate $\SS$
into three regions: One annulus between the two components of $\partial A$,
and, without loss of generality, two disks containing $\{a,b\}$ and
$\{c,d\}$, respectively. (See Figure \ref{F:essann1}$(i)$ as an aid to this
discussion.) The two disks together with $A$ bound a topological ball $B$,
and the annulus together with $A$ bounds a topological solid torus $T$.
Because $A$ is essential, some subset of $\Sigma$ must lie in $T$ (and
similarly for the box marked $R$ in the figure). However, since $\QQ$ is
atoroidal, the only possibility is that one single
loop of $\Sigma$ lies at the core of $T$ (labeled $e$ in the figure).  Notice
that, once $A$ has been identified, it is not possible for another
nonsingular essential annulus (separating, for instance, $\{a,d\}$ from
$\{b,c\}$ on $\SS$) to exist. See Figure \ref{F:only1annulus}. This is
because such an annulus would force, by the same argument above, the
existence of another loop in $\Sigma$ (shown winding from north to south, in
the figure) that would have to be contained in $B$ but \emph{also} contained
in a solid torus (similar to $T$ above, but with one longitudinal annulus in
its boundary contained in $\SS$ and separating $\{a,c\}$ from $\{b,d\}$) that
retracts to $\SS$ without its core curve having to cross through the core of
$T$ (labeled $e$ in the figure).
\begin{figure}
  \labellist
  \small\hair 2pt
  \pinlabel {$a$} [ ] at 50 141
  \pinlabel {$b$} [ ] at 103 141
  \pinlabel {$e$} [ ] at 130 93
  \pinlabel {$R$} [ ] at 75 80
  \pinlabel {$c$} [ ] at 48 16
  \pinlabel {$d$} [ ] at 103 16
\endlabellist
  \includegraphics{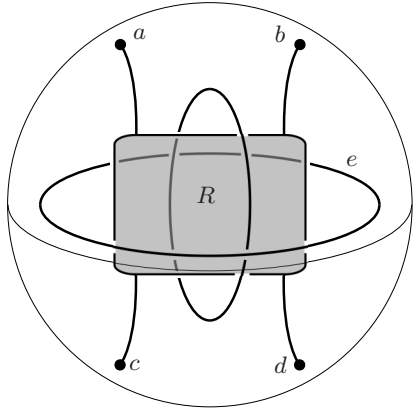}
  \caption{The two closed curves linking the region labeled $R$ can be used
  to show that there can only be one essential nonsingular annulus.}
		\label{F:only1annulus}
	  \end{figure}
We note, further, that at least one integer from each of the pairs $\{a,b\}$
and $\{c,d\}$ must be greater than 2. If not, then it is possible to form an
essential torus in $\QQ$, contradicting the fact that $\QQ$ is atoroidal. See
Figure \ref{F:esstorus1}. This argument also rules out the existence of any
essential singular annulus (i.e., a disk with
one or two order two singular points), and justifies Figure
\ref{F:essann1}$(i)$.

\begin{figure} 
  \labellist
  \small\hair 2pt
  \pinlabel {$2$} [ ] at 65 148
  \pinlabel {$2$} [ ] at 88 148
  \pinlabel {$R$} [ ] at 75 79
  \pinlabel {$c$} [ ] at 64 6
  \pinlabel {$d$} [ ] at 87 6
  \pinlabel {$e$} [ ] at 126 84
\endlabellist
  \includegraphics{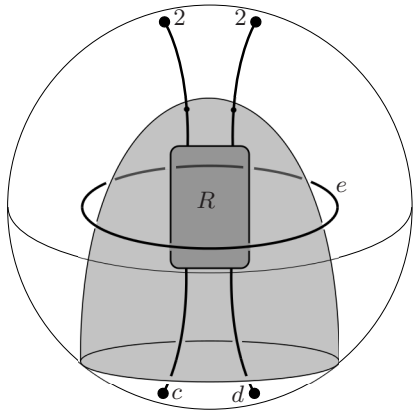}
		  \caption{The indicated annular disk with two singular points
		  determines an essential torus in $\OO$.}		
	\label{F:esstorus1}
\end{figure}

Notice that if the core curve $e$ of $T$ is labeled by 2, then we obtain
another, nonorientable essential annulus, as described in Section
\ref{s:background} (in Figure \ref{F:essann1}$(i)$, with its one boundary
curve the ``equator'' of $\SS$ and its one mirrored edge along the curve
labeled $e$). We note, for later, that the boundary of a regular neighborhood
of this annulus is an essential orientable annulus, as described in the first
part of this proof, that separates the regular neighborhood from the rest of
$\QQ$ as an $I$-bundle. 

If there is no nonsingular annulus, then there may exist up to two essential
annuli, each of which with underlying space a disk and with two cone points
of order 2. Because $\QQ$ is irreducible, the boundary of each such annulus
must separate the cone points of $\SS$ into pairs, say, without loss of
generality, $\{a,b\}$ from $\{c,d\}$. (Note: for topological reasons, just as
in the argument accompanying Figure \ref{F:only1annulus}, if there are two
such annuli, then they must both separate  $\{a,b\}$ from $\{c,d\}$.)
Because $\SS$ is not an orbifold torus (this would violate either its
incompressibility or the fact that $\QQ$ is atoroidal), then supposing such
annuli exist leads to two possibilities: neither of $\{a,b\}$ or $\{c,d\}$ is
equal to $\{2\}$, or one of these sets is equal to $\{2\}$. See Figure
\ref{F:essann1}$(ii)$. Suppose, without loss of generality, that $a$ and $b$
are both 2. Then we are able, by letting $A$ be the annulus separating the
central arc labeled $n$ of the figure from the region $R_{2}$, to form an
orbifold torus by attaching the boundary of a disk in $\SS$ that  contains
$\{a,b\}$ to $\partial A$. Since $\QQ$ is atoroidal, this orbifold torus
would bound a solid orbifold torus, forcing $n=1$ (so that that arc is not
properly a part of the singular locus), the two supposed annuli to be
parallel, and $R_{1}$ to be equal to a solid orbifold torus.  The other case,
in which neither of $\{a,b\}$ or $\{c,d\}$ is equal to $\{2\}$ and $n \geq
1$, may occur.

Note that, in the cases considered in the above paragraph, there can be no
part of the singular locus that surrounds the part of the singular locus in
Figure \ref{F:essann1}$(ii)$, as this would give rise to an essential torus
in $\QQ$ similar to the one illustrated in Figure
\ref{F:esstorus1}.

Finally, there is the possibility that $A$ is nonorientable and has either
one singular point or two corner points (there are two such annuli, listed in
Figure \ref{F:OrbifoldAnnuli}). In this instance, the boundary of a regular
neighborhood of $A$ falls into the category of Figure \ref{F:essann1}$(ii)$,
and just as in the case of a nonorientable annulus in Figure
\ref{F:essann1}$(i)$, the boundary of this regular neighborhood separates the
regular neighborhood from the rest of $\QQ$ as an $I$-bundle. This completes
the proof of the lemma.
\end{proof}

\subsection{Proof of the main theorem}

What follows is a proof of Theorem~\ref{T:VolumeGuts}. Recall that $\OO$ is a compact,
orientable, irreducible, turnover-reduced, $3$--orbifold with underlying
space $S^3$ whose interior admits a hyperbolic structure of finite volume.
Let $\SS$ be a closed, incompressible $2$--suborbifold of $\OO$ of the form 
$\SS = S^2(n_1,n_2,n_3,n_4)$.
 Our goal is a lower volume
bound on $\OO$ in terms of the topology of $\SS$. The plan of the proof is to
find all possibilities for $guts(\SS)$ by
applying Lemma~\ref{l:annuli} and then applying
Theorem~\ref{T:AST}.

\begin{proof}  To begin, we note that $\partial \OO = P \cup
H$, where $P$ is the set of rigid cusp neighborhood boundaries and $H$ the
set of hyperbolic turnovers (corresponding to Euclidean and hyperbolic
vertices in the singular set $\Sigma$ of $\OO$, respectively). The presence
of hyperbolic turnovers in the boundary prevents the interior of $\OO$ from
having a finite volume hyperbolic structure, but $\OO$ does admit a unique
hyperbolic structure of finite volume in which the turnovers are realized as
totally geodesic boundary components. 
Let $D(\OO) = \OO \cup_H \overline{\OO}$ be the double of $\OO$ along its totally geodesic turnover
boundary components where $\overline{\OO}$ denotes $\OO$ with reversed
orientation. The orbifold $\OO$ is naturally a $3$--suborbifold of $D(\OO)$.
Since an incompressible annulus in $\OO$ can be made to be disjoint from any
incompressible turnover in the boundary of $\OO$, $D(\OO)$ has a hyperbolic
structure of finite volume on its interior with possible cusps coming from
Euclidean vertices in the singular set of $\OO$.  Let $G = guts(D(\OO)\split
D(\SS), D(p))$. We then define $guts(\SS)$ to be the intersection of $G$ with
$\OO$. Note that we may apply Theorem~\ref{T:AST} directly to $guts(\SS)$ since
$\vol(G) = 2\vol(guts(\SS))$. 

Let $\QQ$ denote one of the two components of $\OO \split \SS$; it is a compact, orientable, irreducible, atoroidal, turnover-reduced
3-orbifold, with one incompressible boundary component $\SS$ (that is not a
hyperbolic turnover) and all other boundary components consisting of rigid
and hyperbolic turnovers. Note that $\QQ$ satisfies the assumptions of
Lemma~\ref{l:annuli}. We will abuse notation and call the collection of rigid
boundary components $P$. Recall the notation $\partial_{0} \QQ = \partial \QQ
- P$. Because hyperbolic turnovers are always incompressible, $\partial_{0}
\QQ$ is incompressible. Using the remarks after Definition \ref{d:pared}, we
will identify the characteristic subpair and, subsequently, the guts of
$(\QQ, P)$. 

If $(\QQ, \partial \QQ_{0})$ contains no essential
annuli (that are not boundary parallel), then it is acylindrical and admits a
hyperbolic metric with totally geodesic boundary. In particular, it is equal
to its guts, and in this case, Theorem \ref{T:AST} provides a lower bound of
\begin{equation}\label{E:Bound1}
	\Vol(\QQ) \geq -V_{8}\chi(guts(\SS)) = -\frac{V_{8}}{2} \chi(\partial \QQ) \geq -\frac{V_{8}}{2} \chi(\SS).
\end{equation}
The latter inequality takes into account that some components
of the boundary may be hyperbolic turnovers, which we discard in the estimate
because they may or may not be present. Otherwise, $\QQ$ admits one of the
configurations of essential annuli from the Lemma~\ref{l:annuli}. We examine the
cases.

In the case of Figure \ref{F:essann1}$(i)$, the essential annulus cuts off a
solid torus with an annulus in its boundary contained in $\partial \QQ$. The
remaining piece could be pared acylindrical, in which case the Euler
characteristic of its boundary, being the same as that of $\SS$, yields the
same lower volume bound as in (\ref{E:Bound1}). If not, then because it is
not a solid orbifold torus, it must be an ($I$-bundle, $\partial
I$-subbundle) pair as in item \ref{C:char}.(3). In particular, we obtain no
lower volume bound, but we do have a concrete description of the type of
$\QQ$ in this case (and the next). See Proposition \ref{L:eyebundles} below, after
the proof. The case of Figure \ref{F:essann1}$(i)$ when $e=2$ is exactly
analogous to the previous case, with the same possible volume bounds.

In the case of Figure \ref{F:essann1}$(ii)$, there are either one or two
essential annuli, each of which is a  disk with two singular points of order
2. If $n \geq 2$, then there are two annuli, and they divide $\QQ$ into three
components. If $n=1$, then there is only one annulus and it divides $\QQ$
into two components. In the first case, the ``middle'' component is a solid
torus and so contributes nothing to the guts, and the ``upper'' and ``lower''
components may or may not contribute to the guts, depending on whether they
are ($I$-bundle, $\partial I$-subbundle) pairs. If, for example, the upper
component in the figure is not such a bundle, then it provides the following
guts-based lower volume bound contribution, based on the fact that its
boundary is a 2-orbifold $X$ with underlying space the 2-sphere and with four
singular points of orders $a$, $b$, 2, and 2:
\begin{equation}\label{E:Bound2}
	\Vol(\QQ) \geq -\frac{V_{8}}{2} \chi(X) = \frac{V_{8}}{2} \left(1 - \frac{1}{a} - \frac{1}{b} \right).
\end{equation}
A similar computation holds for the lower component. In the case when both of
these components are $I$-bundles, we do not obtain a volume bound, but we do
have a characterization given in Proposition \ref{L:eyebundles} below, after the
proof.  

Combining these results, we see that the volume bounds we obtain from
(\ref{E:Bound1}) and (\ref{E:Bound2}) fall into six categories, corresponding
to the items in Theorem \ref{T:VolumeGuts}.1:
\begin{enumerate}
	\item Both sides of $\SS$ in $\OO$ contribute fully to the guts, so we
	  combine the two bounds of type (\ref{E:Bound1}), or
	\item one side of $\SS$ contributes fully to the guts, and the other side
	  contributes the guts only using two of the four singular points of
	  $\SS$, so we combine the bound of type (\ref{E:Bound1}) with the bound
	  of type (\ref{E:Bound2}), or 
	\item one side of $\SS$ contributes fully to the guts, and the other side
	  contributes nothing to the guts, so we use only one volume bound of
	  type (\ref{E:Bound1}), or
	\item each side of $\SS$ contributes to the guts using only two of the
	  four singular points of $\SS$, and so we combine two volume bounds of
	  type (\ref{E:Bound2}), or
	\item one side of $\SS$ contributes to the guts using only two of the
	  four singular points of $\SS$, and so we use only one volume bound of
	  type (\ref{E:Bound2}), or
	\item neither side of $\SS$ contributes to the guts, and we use Lemma
	  \ref{L:eyebundles} in Section \ref{s:emptyguts} to classify these
	  orbifolds.
\end{enumerate}
This completes the proof of Theorem \ref{T:VolumeGuts}.
\end{proof}

\section{Orbifold $I$--bundles and  hungry orbifolds}\label{s:hungry}

In the proof of the main theorem, the only cases in which a lower bound could
not be obtained involved the possibilities when the regions $R$, $R_{1}$, and
$R_{2}$ in Figure \ref{F:essann1} were ($I$-bundle, $\partial I$-subbundle)
pairs. The following proposition classifies these orbifolds.

\begin{proposition}\label{L:eyebundles}Let $M$ be an orientable 3-orbifold with underlying space a 3-ball such that $\partial M$ has underlying space a 2-sphere with 4 singular points, obtained from the orbifold $\QQ$ as above. If $M$ is an orbifold $I$-bundle, then $M$ has one of the forms given in Figure \ref{F:bundles}.
\end{proposition}
\begin{proof} For the purposes of this lemma, an orbifold $I$-bundle is a
  space of the form $(F \times [-1,1])/((x,y) \sim (\varphi(x),-y))$, where
  $F$ is a connected, orientable 2-orbifold and $\varphi$ is an involution of
  $F$ \cite[page 443]{Bonahon-Siebenmann}. Of course, one possibility is that
  $F$ has two singular points and $\varphi$ is the identity, in which case
  $M$ is just a product of $F$ with an interval, as in Figure
  \ref{F:bundles}$(vi)$. If the underlying space of $F$ has genus greater
  than zero, then such a quotient as above will not have underlying space the
  3-ball, because a component of its boundary will not have the appropriate
  underlying space. If the underlying space of $F$ is the 2-sphere, then
  there are four choices (up to isotopy) for the involution $\varphi$: the
  identity, a reflection across a great circle, a rotation of order two, or the antipodal map. 
\begin{figure}
	\labellist
	\small\hair 2pt
	\pinlabel {$2$} [ ] at 130 265
	\pinlabel {$a$} [ ] at 40 325
	\pinlabel {$b$} [ ] at 112 325
	\pinlabel {$a$} [ ] at 213 325
	\pinlabel {$a$} [ ] at 293 325
	\pinlabel {$2$} [ ] at 281 261
	\pinlabel {$2$} [ ] at 251 236
	\pinlabel {$b$} [ ] at 213 211
	\pinlabel {$c$} [ ] at 293 211
	\pinlabel {$a$} [ ] at 397 312
	\pinlabel {$b$} [ ] at 475 312
	\pinlabel {$2$} [ ] at 429 290
	\pinlabel {$2$} [ ] at 409 270
	\pinlabel {$2$} [ ] at 451 270
	\pinlabel {$2$} [ ] at 429 260
	\pinlabel {$d$} [ ] at 395 225
	\pinlabel {$c$} [ ] at 462 225
	\pinlabel {$a$} [ ] at 50 135
	\pinlabel {$a$} [ ] at 103 135
	\pinlabel {$2$} [ ] at 50 21
	\pinlabel {$2$} [ ] at 102 21
	\pinlabel {$a$} [ ] at 228 138
	\pinlabel {$b$} [ ] at 278 138
	\pinlabel {$2$} [ ] at 208 74
	\pinlabel {$2$} [ ] at 252 97
	\pinlabel {$2$} [ ] at 297 74
	\pinlabel {$a$} [ ] at 406 138
	\pinlabel {$b$} [ ] at 466 138
	\pinlabel {$a$} [ ] at 406 18
	\pinlabel {$b$} [ ] at 466 18
	\pinlabel {$(i)$} [ ] at 75 182
	\pinlabel {$(ii)$} [ ] at 252 182
	\pinlabel {$(iii)$} [ ] at 431 182
	\pinlabel {$(iv)$} [ ] at 75 -8
	\pinlabel {$(v)$} [ ] at 253 -8
	\pinlabel {$(vi)$} [ ] at 430 -8
  \endlabellist
  \centering
  \includegraphics[scale=0.75]{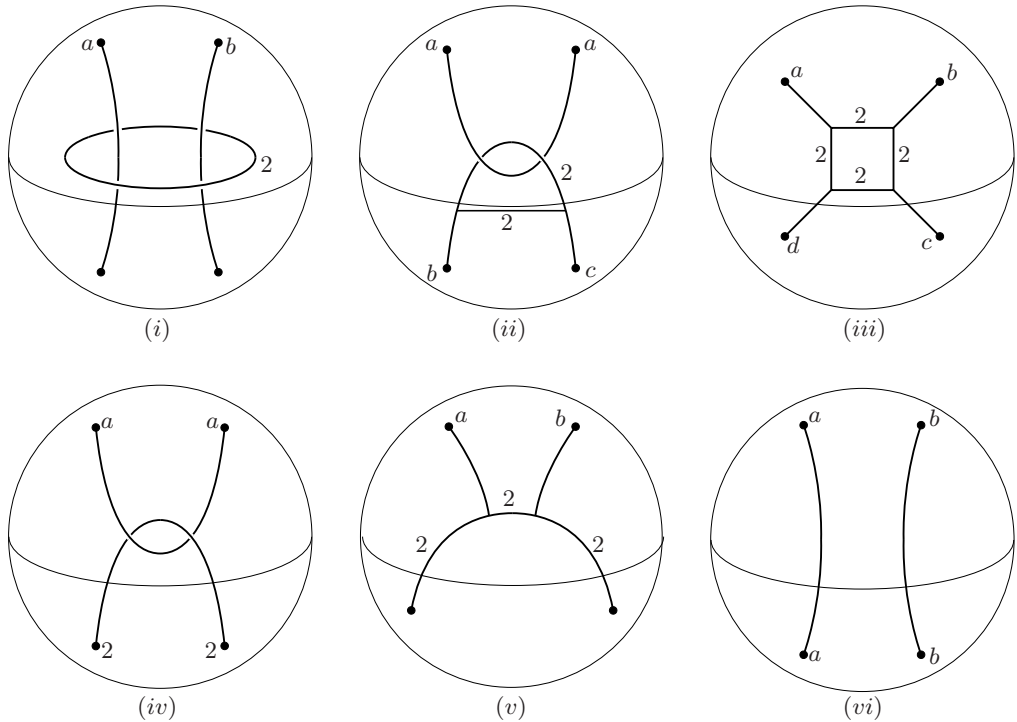}
  \caption{The different possibilities for $I$-bundles.}
  \label{F:bundles}
\end{figure}

\begin{figure}
\labellist
\small\hair 2pt
\pinlabel {$a$} [ ] at 10 152
\pinlabel {$a$} [ ] at 77 152
\pinlabel {$180^\circ$} [ ] at 40 56
\pinlabel {$F\times\{1\}$} [ ] at -20 151
\pinlabel {$F\times\{0\}$} [ ] at -20 82
\pinlabel {$F\times\{-1\}$} [ ] at -24 14
\pinlabel {$2$} [ ] at 180 119
\pinlabel {$a$} [ ] at 199 127
\pinlabel {$2$} [ ] at 218 119
\pinlabel {$\longrightarrow$} [ ] at 126 83
\endlabellist
\centering
\includegraphics[scale=1.0]{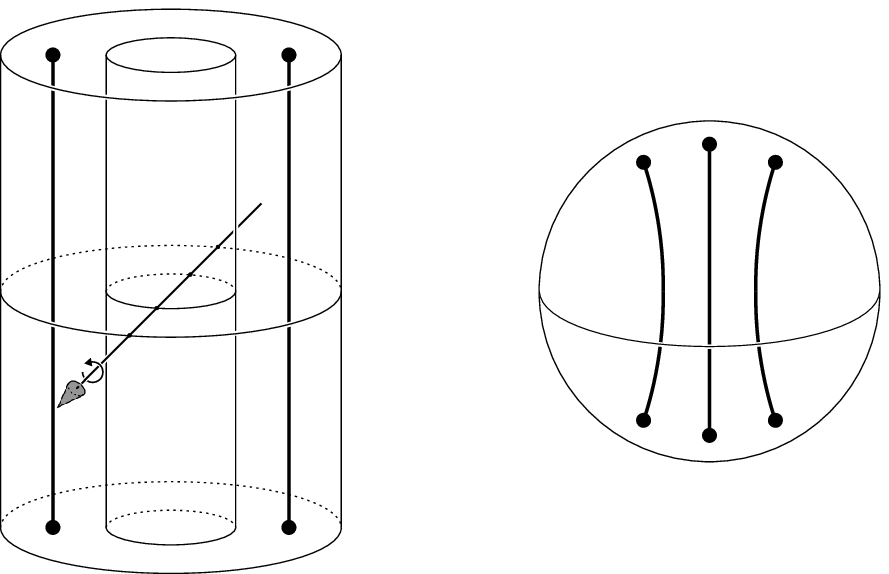}
\caption{Symmetry of a surface crossed with an interval, yielding an $I$-bundle.}
\label{F:toomanypoints}
\end{figure}
Without even needing to consider the singular locus, we note that the
identity map, the antipodal map, and a rotation of order two will each yield
a quotient space either whose underlying space is not the 3-ball or that
contains a 2-orbifold that  cannot occur in $\QQ$ (by the construction of
$\QQ$, as coming from the 3-sphere with an embedded trivalent graph).
However, if $F$ has underlying space the 2-sphere and if $\varphi$ is a
reflection across a great circle, then the quotient has underlying space a
3-ball and contains one unknotted loop labeled 2 in its interior. In order to
obtain a quotient whose boundary contains four singular points, we have three
options: (1) If $F$ has four singular points of orders $a$, $a$, $b$, and
$b$, and $\varphi$ is a reflection that interchanges the singular points of
corresponding orders, then the resulting quotient is depicted in Figure
\ref{F:bundles}$(i)$; (2) If $F$ has four singular points of orders
$a$, $a$, $b$, and $c$, and $\varphi$ fixes only the singular points labeled
$b$ and $c$ and interchanges the singular points labeled $a$, then the result
is depicted in Figure \ref{F:bundles}$(ii)$; (3) If $F$ has four singular
points of orders $a$, $b$, $c$, and $d$ and $\varphi$ fixes them all, then
the result is Figure \ref{F:bundles}$(iii)$. 

If $F$ has genus zero and multiple boundary components, then an involution that interchanges two boundary components will create higher genus in the boundary of the quotient, so $\varphi$ must leave the boundary circles of $F$ invariant. There are only three possibilities for how an involution of $F$ can act on a boundary circle: the identity, a reflection, or the antipodal map. If $F$ has more than one boundary circle, then the identity involution will create higher genus in the boundary of the quotient, and any action by the antipodal map will yield a quotient with a non-orientable crosscap in its boundary, contrary to our assumption on $M$. So we are left with a reflection, which extends from the boundary circles of $F$ to a reflection of the whole orbifold $F$.

If $\varphi$ is a reflection of a 2-orbifold $F$ with genus zero, no singular points, and  two  boundary components, then the $I$-bundle quotient obtained from $\varphi$ has four singular points of order two in its boundary. See Figure \ref{F:toomanypoints} (with $a=1$). If $F$ has more than two boundary components, or at least two boundary components and some singular points, then the resulting quotient will have more than four singular points in its boundary. See Figure \ref{F:toomanypoints} (with $a > 1$) in the case that $F$ has two boundary components and one singular point.
In the cases that we are considering, this is not possible. So $F$ must have a single boundary component. Because $\varphi$ is a reflection, we have only two possibilities: (1) If $F$ has two singular points of the same order that are exchanged, then the result is Figure \ref{F:bundles}$(iv)$; (2) If $F$ has two singular points labeled $a$ and $b$ that are fixed by $\varphi$, then the result is Figure \ref{F:bundles}$(iv)$. This completes the proof of the lemma. 
\end{proof}

\subsection{The case of hungry orbifolds}\label{s:emptyguts}

When the guts of $\OO\split\SS$ are empty, then $\OO$ can have only one of finitely
many forms of a certain type. We describe these types in this section here.

A \emph{rational tangle operation} on a 2-sphere with four marked points is
an isotopy of the 2-sphere that permutes the four marked points.

Let $\TT$ denote $S^2$ with four marked points. Let $\sigma: \TT\times [0,1]
\to \TT$ be an isotopy and let $\sigma_t(x) = \sigma(x,t)$ for $x \in \TT$
and $t \in [0,1]$. The \textit{isotopy cylinder} for $\sigma$ is the set
\[\TT_{\sigma} = \left\{ (\sigma_t(x) , t) \, \mid \, x \in \TT, t \in [0,1]
\right\}.\] Note that in $\TT_{\sigma}$, the marked points trace out a braid
in $\TT \times I$.

Suppose $\QQ_i$ is an orbifold with $\TT_i \subseteq \partial \QQ_i$ for $i
\in \{0,1\}$ where each $\TT_i$ is homeomorphic to $\TT$. If $\sigma$ is a
rational tangle operation, then we can glue $\QQ_0$ to $\QQ_1$ along the
isotopy cylinder for $\sigma$ to obtain $\QQ_0 \sqcup_{\sigma} \QQ_1$. More
precisely, 
\[\QQ_0 \sqcup_{\sigma} \QQ_1 = (\QQ_0 \amalg \QQ_1 \amalg \TT_\sigma)/\sim\]
where $\sim$ is an identification such that $x \sim (x,i)$ for each $x$ in
$\TT_i$ and $i\in \{0,1\}$.

\begin{proposition}\label{P:hungry}
Let $\OO$ be a $3$--orbifold containing an incompressible $2$--suborbifold
$\SS$ of the form $S^2(n_1,n_2,n_3,n_4)$ or $D^2_*(n_1,n_2)$. Suppose that
$guts(\SS)$ is empty. Then $\OO$ has one of the following forms:
\begin{enumerate}
\item If $\SS = S^2(n_1,n_2,n_3,n_4)$, then $\OO = \QQ_0 \sqcup_\sigma \QQ_1$
where each of $\QQ_i$, $i\in\{0,1\}$ is one of the orbifolds in
Figure~\ref{F:essann1} with each of $R$, $R_1$, and $R_2$ equal to one of the
bundles from Figure~\ref{F:bundles}, and $\sigma$ a rational tangle isotopy.

\item If $\SS = D^2_*(n_1,n_2)$, then $\OO = \QQ_0 \sqcup_\sigma \QQ_1$
where $\QQ_0$, is one of the orbifolds in
Figure~\ref{F:essann1} with each of $R$, $R_1$, and $R_2$ equal to one of the
bundles from Figure~\ref{F:bundles} and $\QQ_1$ is a regular neighborhood of
$\SS$ in $\OO$, and $\sigma$ a
rational tangle isotopy.
\end{enumerate}
\end{proposition}

\begin{proof}
Suppose that $\SS = S^2(n_1,n_2,n_3,n_4)$. Since $\SS$ is orientable, $\OO
\setminus \SS$ has two components. Denote these components by $\QQ_0$ and
$\QQ_1$. If $\QQ_i$ contains an incompressible annulus, then by
Lemma~\ref{l:annuli}, $\QQ_i$ and its annulus are configured as in
Figure~\ref{F:essann1}. By the characteristic suborbifold theory, each of the
regions $R$, $R_1$, and $R_2$ in these cases must actually be an orbifold
$I$--bundle (see Figure~\ref{F:bundles}), for otherwise, the portion of
$\QQ_i$ on that side of the essential annulus would contradict the fact that
$guts(\SS)$ is empty.

If $\QQ_i$ contains no essential annuli, then since $guts(\SS)$ is empty, by
the characteristic suborbifold theory, $\QQ_i$ must be an $I$--bundle. By
Proposition~\ref{L:eyebundles}, $\QQ_i$ must be of one of the six types
described by Figure~\ref{F:bundles}. By inspection, observe that $I$--bundles
of types ($i$), ($ii$), and ($iii$) of Figure~\ref{F:bundles} contain an
essential orbifold annulus. Therefore $\QQ_i$ must be an $I$--bundle of type
($iv$), ($v$), or ($vi$). However, each of these types have compressible
boundaries, contradicting the assumption that $\SS$ is incompressible in
$\OO$.

Let $\QQ_i$, $i\in\{0,1\}$ be two orbifolds as described in (1) of the
statement of the proposition. If $\sigma$ is any rational tangle operation on
$\SS$, then since $\sigma$ is an isotopy, the two components of $\QQ_0
\sqcup_\sigma \QQ_1 \split \SS$  have the same orbifold type as $\QQ_0$ and
$\QQ_1$.

If $\SS = D_*^2(n_1,n_2)$, then let $\QQ_1$ be a closed regular neighborhood
of $\SS$ in $\OO$. Then the boundary of $\QQ_1$ has the form
$S^2(n_1,n_1,n_2,n_2)$. The same argument as in the orientable case can be
used to show that $\OO = \QQ_0 \sqcup_\sigma \QQ_1$ where $\QQ_0$ is as
described in the statement of the proposition and $\sigma$ is a rational
tangle operation on  $S^2(n_1,n_1,n_2,n_2)$.
\end{proof}

\bibliographystyle{hamsplain}
\bibliography{Gutsamoto}

\end{document}